\documentclass[11pt,a4paper]{article}

\usepackage{amssymb,amsmath,amsfonts,amsthm,mathrsfs,mathtools}
\usepackage{bbm,bm}
\usepackage[marginal]{footmisc}
\usepackage{CJK}
\usepackage{mathrsfs}
\usepackage{graphicx}
\usepackage{float}
\usepackage{xcolor}
\usepackage{url}
\usepackage{relsize}
\usepackage{appendix}
\usepackage{epstopdf}
\usepackage{enumitem}

\usepackage[colorlinks,linkcolor=blue,anchorcolor=blue,citecolor=purple,urlcolor=green]{hyperref}

\title{\bf  The limit of the optimal Li--Yau constant \\ for the fractional heat equation}

\author{Huaiqian Li\footnote{Email: {\color{blue}huaiqian.li@tju.edu.cn}\quad  Partially supported by the National Key R\&D Program of China (Grant No. 2022YFA1006000) and  the National Natural Science Foundation of China (Grant No. 12671176).}
\vspace{2mm}
\\
{\footnotesize Center for Applied Mathematics and KL-AAGDM, Tianjin University,}\\ {\footnotesize Tianjin 300072, China}
}

\date{}

\numberwithin{equation}{section}

\newtheorem{theorem}{Theorem}[section]
\newtheorem{proposition}[theorem]{Proposition}
\newtheorem{lemma}[theorem]{Lemma}

\theoremstyle{remark}
\newtheorem{remark}[theorem]{Remark}

\newcommand{\R}{\mathbb{R}}
\newcommand{\E}{\mathbb{E}}
\newcommand{\Pb}{\mathbb{P}}
\newcommand{\dd}{\,\mathrm{d}}
\newcommand{\CLY}{C_{\mathrm{LY}}}
\newcommand{\Phib}{\Phi_\beta}
\newcommand{\psib}{\psi_\beta}
\newcommand{\eps}{\varepsilon}
\newcommand{\sphere}{\mathbb{S}^{d-1}}
\newcommand{\e}{\mathbf{e}}

\begin{document}
\allowdisplaybreaks
\maketitle
\makeatletter 
\renewcommand\theequation{\thesection.\arabic{equation}}
\@addtoreset{equation}{section}
\makeatother 

\begin{abstract}
For the fractional heat equation $\partial_t u + (-\Delta)^{\beta/2}u=0$ on $\mathbb{R}^d$ with $\beta\in(0,2)$, Weber and Zacher introduced the optimal constant $C_{\rm LY}(\beta,d)$ in the Li--Yau type estimate
$$(-\Delta)^{\beta/2}\log u(t,\cdot)\leq \frac {C_{\rm LY}(\beta,d)}{t},\quad t>0.$$
They asked whether $C_{\rm LY}(\beta,d)$ converges to the classical Li--Yau constant $d/2$ as $\beta\uparrow2$.  In this note we prove that the answer is affirmative.  More precisely, for every fixed $d\geq1$ there exists a constant $K_d>0$ such that
$$\frac{d}{\beta}\leq  C_{\rm LY}(\beta,d)\leq\frac{d}{\beta}+K_d(2-\beta)\log\frac{e}{2-\beta},\quad 1\leq\beta<2.$$
Consequently, $\lim_{\beta\uparrow2}C_{\rm LY}(\beta,d)=d/2$.  The proof has two main ingredients.  First, the Bochner subordination formula, together with a concavity argument, shows that the supremum in Weber--Zacher's formula for $C_{\rm LY}(\beta,d)$ is attained at the origin. Second, the factor $2-\beta$ is derived from a lower bound for the fractional heat kernel profile in the large-distance regime, where the L\'{e}vy--It\^{o} decomposition plays an important role.
\end{abstract}



\section{Introduction}\hskip\parindent
The classical Li--Yau inequality states that, on a complete $d$-dimensional Riemannian manifold $\mathbb{M}$ with non-negative Ricci curvature, positive solutions of the heat equation $\partial_t u=\Delta u$ on $(0,\infty)\times \mathbb{M}$ satisfy a sharp differential Harnack estimate, i.e.,
  \begin{eqnarray}\label{LY86}
-\Delta\big(\log u(t,\cdot)\big)(x)\leq\frac{d}{2}t^{-1},\quad t>0,\, x\in \mathbb{M},
\end{eqnarray}
where $\Delta$ is the Laplace--Beltrami operator. In $\R^d$, the optimal constant $d/2$ is realized by the standard heat kernel
\begin{eqnarray}\label{Gauss-hk}
p_t(x)=\frac{1}{(4\pi t)^{d/2}}\exp\Big(-\frac{|x|^2}{4t}\Big),\quad x\in\R^d,\,t>0.
\end{eqnarray}
Refer to the seminal paper of Li and Yau \cite{LiYau1986}.

Recent work has aimed at finding analogues of Li--Yau estimates for non-local diffusion operators.  Weber and Zacher \cite{WeberZacher2023} established a general reduction principle, applicable in discrete and continuous settings, which derives non-local Li--Yau inequalities from corresponding heat kernel estimates. A key ingredient in their approach is a convexity estimate for the non-local quantity $\Psi_\Upsilon(\log u)$ (see equation (3) in the aforementioned paper), which allows heat kernel estimates to be transferred to arbitrary positive solutions. As a main application, they considered the fractional heat equation on $\R^d$:
\begin{equation}\label{fractional-HE}
\partial_t u + (-\Delta)^{\beta/2}u=0,\quad 0<\beta<2.
\end{equation}
They proved that every positive strong solution $u$ of \eqref{fractional-HE} satisfies
\begin{equation*}\label{WZ-LY}
 (-\Delta)^{\beta/2}\big(\log u(t,\cdot)\big)(x) \leq \frac{\CLY(\beta,d)}{t},\quad t>0,\,x\in\R^d,
\end{equation*}
where the constant $\CLY(\beta,d)$ is determined by the fractional heat kernel and is optimal among all constants for which the corresponding estimate holds for the heat kernel itself. Here, we say that $u$ is a strong solution to \eqref{fractional-HE} if for every fixed $T>0$,
$$u\in C([0,T)\times \R^d),\quad \partial_t u\in C((0,T)\times\R^d),$$
and \eqref{fractional-HE} holds pointwise for all $(t,x)\in(0,T)\times\R^d$. In \cite[Remark 3.4]{WeberZacher2023}, they asked the following question:
\begin{itemize}[leftmargin=2cm]
\item[($\mathcal{Q}$)] Does the optimal constant $\CLY(\beta,d)$ converge to the classical Li--Yau constant $d/2$ as $\beta\to2^-$?
\end{itemize}

Sharp Li--Yau inequalities in other non-local frameworks often depend on fine properties of the heat kernel. For instance, in \cite{LiQian2023}, the author and Qian proved sharp Li--Yau inequalities for the heat equation associated with Dunkl harmonic oscillators, which are non-local Schr\"odinger-type operators involving reflection symmetries and multiplicity functions.  That work was motivated by \cite{WeberZacher2023}, where sharpness is again obtained by reduction to the heat kernel. Discrete counterparts include the sharp Li--Yau inequality obtained in \cite[Theorem~4.1]{WeberZacher2023} for the continuous-time Markov chain associated with a finite complete graph, as well as the sharp semi-ultra-log-convexity estimate for the $\mathrm{M}/\mathrm{M}/\infty$ semigroup established by Chen and the author in \cite{ChenLi2025} through a similar transition-kernel argument. These results demonstrate how sharp differential or discrete convexity estimates for an entire semigroup can be obtained by identifying and exploiting the relevant structural properties of its heat or transition kernel.  Building on this, the present paper focuses on the limiting problem for $\CLY(\beta,d)$ as $\beta\to2^-$ and answers the above question ($\mathcal{Q}$) affirmatively.

Let $\beta\in(0,2)$. We use the same normalization of the fractional Laplacian as in \cite{WeberZacher2023} (see \cite{Garofalo2019,Kwas2017} for other equivalent formulations). For any sufficiently regular function $f$,
\begin{align}\label{Abeta}
L_\beta f(x)&:=(-\Delta)^{\beta/2}f(x)\cr
&=\frac{c_{\beta,d}}{2}\int_{\R^d}\frac{2f(x)-f(x+h)-f(x-h)}{|h|^{d+\beta}}\dd h,\quad x\in\R^d,
\end{align}
where $c_{\beta,d}$ is a normalizing constant with precise value given by
\begin{equation*}\label{cbetad}
c_{\beta,d}=\frac{2^\beta\Gamma\bigl((d+\beta)/2\bigr)}{\pi^{d/2}|\Gamma(-\beta/2)|},
\end{equation*}
and $\Gamma$ denotes the Gamma function. A detailed derivation of $c_{\beta,d}$ can be found in \cite[Proposition 5.6]{Garofalo2019} and \cite[Theorem 2.1]{BlumenthalGetoor1960}. Let $G^{(\beta)}(t,x)$ be the fractional heat kernel associated with $L_\beta$. It is well known that $G^{(\beta)}(\cdot,\cdot)$ is positive, infinitely differentiable on $(0,\infty)\times\R^d$, and for $t>0$, $G^{(\beta)}(t,\cdot)$ is radial and $\int_{\R^d} G^{(\beta)}(t,x)\dd x=1$. We use the following normalization
$$\widehat{G^{(\beta)}(t,\cdot)}(\xi)=e^{-t|\xi|^\beta},\quad t>0,\,\xi\in\R^d,$$
where $\widehat{f}$ stands for the Fourier transform of $f$. Letting
\begin{equation*}
\Phib(x):=G^{(\beta)}(1,x),
\end{equation*}
we have the self-similar form
\begin{equation}\label{self-sim}
       G^{(\beta)}(t,x)=t^{-d/\beta}\Phib(t^{-1/\beta}x),
\end{equation}
where $(t,x)\in(0,\infty)\times\R^d$. For more properties on the fractional heat kernel $G^{(\beta)}(t,x)$, refer to \cite[Sections 16 and 18]{Garofalo2019} and \cite[Section 3]{WeberZacher2023}  for instance.

Weber and Zacher's constant is given by
\begin{equation}\label{CLY-def}
\CLY(\beta,d)=\frac{c_{\beta,d}}{2}\sup_{y\in\R^d}\int_{\R^d}\frac{\log\dfrac{\Phib(y)^2}{\Phib(y+\sigma)\Phib(y-\sigma)}}{|\sigma|^{d+\beta}}\dd\sigma .
\end{equation}
Equivalently (by \eqref{Abeta}),
\begin{equation}\label{CLY-Abeta}
       \CLY(\beta,d)=\sup_{y\in\R^d} L_\beta(\log\Phib)(y).
\end{equation}

Our main result is the following.
\begin{theorem}\label{main}
For every fixed integer $d\geq1$,
\begin{equation}\label{main-limit}
\lim_{\beta\uparrow2}\CLY(\beta,d)=\frac d2.
\end{equation}
More precisely, there exists a constant $K_d>0$ (depending only on $d$) such that, for every $1\leq\beta<2$,
\begin{equation}\label{main-rate}
       \frac d\beta< \CLY(\beta,d)\leq\frac d\beta+K_d(2-\beta)\log\frac e{2-\beta}.
\end{equation}
\end{theorem}

\begin{remark}
(1) The choice of $1$ as the lower endpoint of the range of $\beta$ in \eqref{main-rate} is not essential. More precisely, for every fixed $\beta_{0}\in(0,2)$, there exists a constant $K_{d,\beta_{0}}>0$, depending only on $d$ and $\beta_{0}$, such that
$$\frac{d}{\beta}<C_{\mathrm{LY}}(\beta,d)\le\frac{d}{\beta}+K_{d,\beta_{0}}(2-\beta)\log\frac{e}{2-\beta},\quad\beta_{0}\leq\beta<2.$$
Indeed, this can be proved by the same approach as in the proof of Theorem \ref{main}, provided that the corresponding constants are allowed to depend additionally on $\beta_{0}$.  We chosen $\beta_{0}=1$ only for convenience, which is sufficient for studying the limit as $\beta\to2^-$. We also note that,  by Lemma \ref{main-decomposition}, the strict lower bound
$$C_{\mathrm{LY}}(\beta,d)>\frac{d}{\beta}$$
holds for every $0<\beta<2$.

(2) The logarithmic factor in \eqref{main-rate} comes from the fact that the polynomial tail coefficient of the stable heat kernel degenerates like $2-\beta$ as $\beta\to2^-$.
\end{remark}

\textbf{Structure.} The remainder of the paper is organized as follows. In Section~\ref{sec-preparations}, we develop the preliminary results and keys estimates needed for the proof. Using the Bochner subordination representation, we first establish concavity, subadditivity, and uniform growth properties of the logarithmic radial profile of the fractional heat kernel. We then prove that the supremum in the Weber--Zacher formula \eqref{CLY-def} is attained at the origin. Moreover, particular attention is paid to a large-distance lower bound for the fractional heat kernel that retains the crucial factor $2-\beta$. In Section~\ref{sec-proof}, these estimates are combined to prove Theorem~\ref{main}.

\textbf{Notation.} Throughout the paper, $d\geq1$ is fixed. $\R^d$ is the $d$-dimensional Euclidean space with inner product $\langle\cdot,\cdot\rangle$ and norm $|\cdot|$. Let $\sphere$ denote the unit sphere in $\R^d$, and let $|\sphere|$ denote its volume. For a set $A$, let $\mathbf{1}_A$ denote its indicator function. For $\beta\in(0,2)$, set
$$\alpha=\frac\beta2\in(0,1).$$

\section{Preparations and key estimates}\label{sec-preparations}\hskip\parindent
Let $S^\alpha=(S_t^\alpha)_{t\geq0}$ be the standard $\alpha$-stable subordinator,  i.e., a non-decreasing c\'{a}dl\'{a}g process, defined on some probability space $(\Omega,\mathcal{F},\Pb)$, taking values in $[0, \infty)$, with stationary, independent increments,  and $S_0^\alpha=0$ almost surely, whose Laplace exponent is $\lambda\mapsto \lambda^\alpha$, i.e,
\begin{equation}\label{laplace-subordinator}
       \E e^{-\lambda S_t^\alpha}=e^{-t\lambda^\alpha},\quad t,\lambda\geq0.
\end{equation}
see e.g. \cite[(1.2)]{LiWu2026}. The L\'{e}vy--Khintchine representation of the Bernstein function $\lambda\mapsto \lambda^\alpha$ (see e.g. \cite[(1.1)]{LiWu2026}) is
\begin{equation}\label{LK}
\lambda^\alpha= \int_0^\infty(1-e^{-\lambda y})\,\nu_\alpha({\rm d}y),\quad \lambda\geq0,
\end{equation}
where $\nu_\alpha$ is the L\'{e}vy measure given by
\begin{align}\label{levy-mass}
\dd\nu_\alpha(y)=-\frac{1}{\Gamma(-\alpha)}y^{-1-\alpha}\mathbf{1}_{(0,\infty)}(y)\dd y.
\end{align}
For general treatments of subordinators and L\'evy processes, refer to the monographs \cite{SSV2012,Appl2009,Sato1999}.

We next calculate the negative moments of $S_1^\alpha$. First, by \eqref{laplace-subordinator} and the dominated convergence theorem,
$$\Pb(S_1^\alpha=0) =\lim_{\lambda\to\infty} \E e^{-\lambda S_1^\alpha}=\lim_{\lambda\to\infty}e^{-\lambda^\alpha} =0.$$
Let $\kappa>0$. Using the well-known identity
$$b^{-\kappa} =\frac{1}{\Gamma(\kappa)}\int_0^\infty r^{\kappa-1}e^{-rb}\dd r,\quad b>0,$$
together with the Fubini--Tonelli theorem and \eqref{laplace-subordinator}, we obtain
\begin{align}\label{negative-moment}
\E\bigl(S_1^\alpha\bigr)^{-\kappa}&=\frac{1}{\Gamma(\kappa)}\int_0^\infty r^{\kappa-1}\E e^{-rS_1^\alpha}\dd r\cr
&=\frac{1}{\Gamma(\kappa)}\int_0^\infty r^{\kappa-1}e^{-r^\alpha}\dd r\cr
&=\frac{\Gamma(\kappa/\alpha)}{\alpha\Gamma(\kappa)}<\infty.
\end{align}
Let $\Pb(S_t^\alpha\in\cdot)$ denote the law of $S_t^\alpha$. The Bochner subordination formula gives
\begin{align}\label{subordination}
\Phib(x)&=\int_0^\infty p_s(x)\,\Pb(S_1^\alpha\in{\rm d} s)\cr
&=\E\left[(4\pi S_1^\alpha)^{-d/2}\exp\Big(-\frac{|x|^2}{4S_1^\alpha}\Big)\right],
\end{align}
where $(p_t)_{t>0}$ is the heat kernel given by \eqref{Gauss-hk}. This formula is standard; see \cite[Theorem 18.3]{Garofalo2019} for the heat kernel formulation, \cite[Chapter~13]{SSV2012} for the semigroup formulation, and \cite[Section~30]{Sato1999} for subordinated  L\'evy processes.

Set
\begin{align*}\label{F-beta}
F_\beta(r):=\Phib(\sqrt{r} \e),\quad r\geq0,
\end{align*}
where $\e$ is any fixed unit vector in $\R^d$. By \eqref{subordination},
\begin{equation}\label{F-Laplace}
       F_\beta(r)=\int_0^\infty e^{-r/(4s)}(4\pi s)^{-d/2}\,\Pb(S_1^\alpha\in{\rm d} s),\quad r\geq0.
\end{equation}
Define
\begin{equation}\label{psi-def}
       \psib(r)=\log\frac{F_\beta(0)}{F_\beta(r)},\quad r\geq0,
\end{equation}
Since $\Phib$ is radial, the definition of $\psib$ is independent of the chosen unit vector.

We also need the following properties of the Gamma function $\Gamma$; see e.g.  \cite[Section 4]{Garofalo2019} and \cite{Lebe72} for more details.
\begin{align}\label{gamma}
&\Gamma(1+t)=t\Gamma(t),\quad t>0\cr
&\Gamma(1-t)=-t\Gamma(-t),\quad 0<t<1,
\end{align}
 and moreover, $\Gamma$ is infinitely differentiable on $(0,\infty)$.

We first present some key properties of $\psib$ in the following lemma.
\begin{lemma}\label{psi-concave}
For every $\beta\in(0,2)$, the function $\psib$ has the following properties.
\begin{enumerate}[label=\textup{(\roman*)}]
\item $\psib(0)=0$.

\item $\psib$ is non-negative and non-decreasing on $[0,\infty)$.

\item $\psib$ is concave on $[0,\infty)$.

\item $\psib$ is subadditive, that is,
$$\psib(r+s)\leq \psib(r)+\psib(s),\quad r,s\geq0.$$

\item There exists a constant $M_d>0$, depending only on $d$, such that
\begin{equation*}\label{psi-lip-1}
       0\leq\psib(r)\leq M_d r,\quad r\geq0,\,1\leq\beta<2.
\end{equation*}
\end{enumerate}
\end{lemma}
\begin{proof}
Clearly, (i) is true. From \eqref{F-Laplace}, we easily see that $F_\beta$ is non-increasing, and hence,
$$\psib(r)=\log F_\beta(0)-\log F_\beta(r)$$
 is non-decreasing. Thus, $\psib(r)\geq\psib(0)=0$ for any $r\geq0$. This proves (ii).

We turn to show (iii).  Let $0\leq\theta\leq1$ and $r_1,r_2\geq0$.  By H\"older's inequality, we have
\begin{align*}
&F_\beta(\theta r_1+(1-\theta)r_2)\\
&=\int_0^\infty e^{-\frac{\theta r_1}{4s}} e^{-\frac{(1-\theta)r_2}{4s}}(4\pi s)^{-d/2}\,\Pb(S_1^\alpha\in{\rm d} s) \\
&\leq\left(\int_0^\infty e^{-\frac{ r_1}{4s}}  (4\pi s)^{-d/2}\,\Pb(S_1^\alpha\in{\rm d} s)\right)^\theta
       \left(\int_0^\infty  e^{-\frac{r_2}{4s}}(4\pi s)^{-d/2}\,\Pb(S_1^\alpha\in{\rm d} s)\right)^{1-\theta}  \\
&=F_\beta(r_1)^\theta F_\beta(r_2)^{1-\theta}.
\end{align*}
Thus, $\log F_\beta$ is convex, which immediately implies that $\psib$ is concave.

Now we prove (iv).  Let $r,s>0$. Since $\psib$ is a non-negative concave function on $[0,\infty)$ with $\psib(0)=0$, $r\mapsto \psib(r)/r$ is non-increasing on $(0,\infty)$.  Thus,
$$\frac{\psib(r)}{r}\geq \frac{\psib(r+s)}{r+s},\quad\frac{\psib(s)}{s}\geq \frac{\psib(r+s)}{r+s}.$$
Multiplying by $r$ and $s$, respectively, and adding, we have
$$\psib(r)+\psib(s)\geq \psib(r+s).$$
The cases $r=0$ or $s=0$ are immediate.

It remains only to prove (v).  By differentiating \eqref{psi-def},
\begin{align*}
       \psib'(0+)= -\frac{F_\beta'(0+)}{F_\beta(0)} &=\frac14\frac{\E (S_1^\alpha)^{-d/2-1}}{\E (S_1^\alpha)^{-d/2}}\\
       &=\frac{\Gamma((d+2)/\beta)}{2d\,\Gamma(d/\beta)},
\end{align*}
where the existence of the right derivative of $F_\beta'(0+)$ is an easy consequence of \eqref{negative-moment} and the dominated convergence theorem. As a function of $\beta$, the last term is continuous on the closed interval $[1,2]$, and hence, it is bounded for $1\leq\beta<2$.  Since $\psib$ is concave, non-negative, and $\psib(0)=0$, we have
$$\psib(r)\leq \psib'(0+)r,\quad r\geq0.$$
Letting $M_d=\sup_{\beta\in[1,2)}\frac{\Gamma((d+2)/\beta)}{2d\,\Gamma(d/\beta)}$ finishes the proof.
\end{proof}

The next result is important.  It shows that the supremum in \eqref{CLY-def} is  attained at the origin for every $0<\beta<2$. For $\beta=1$, this was verified directly by using the explicit Cauchy heat kernel in \cite[Proposition~3.3]{WeberZacher2023}.  Proposition \ref{sup-origin} proves the same fact for every $0<\beta<2$ by using subordination and concavity.
\begin{proposition}\label{sup-origin}
For every $0<\beta<2$,
\begin{equation}\label{sup-origin-1}
\CLY(\beta,d)=L_\beta(\log\Phib)(0) =c_{\beta,d}\int_{\R^d}\frac{\psib(|h|^2)}{|h|^{d+\beta}}\dd h.
\end{equation}
\end{proposition}
\begin{proof}
Fix $y,h\in\R^d$. Observe that $\Phib$ is radial and by \eqref{psi-def},
\begin{align}\label{eq-Phi-psi}
\log\Phib(z)=\log\Phib(0)-\psib(|z|^2),\quad z\in\R^d.
\end{align}
Then
\begin{align*}
 &2\log\Phib(y)-\log\Phib(y+h)-\log\Phib(y-h)\cr
 &=\psib(|y+h|^2)+\psib(|y-h|^2)-2\psib(|y|^2)\cr
 &=\psib(|y|^2+ |h|^2+ 2\langle y,h\rangle)+\psib(|y|^2+|h|^2-2\langle y,h\rangle)-2\psib(|y|^2).
\end{align*}
Combining this, by the concavity and subadditivity of $\psib$ (see Lemma \ref{psi-concave} (iii)--(iv)), we obtain
\begin{align*}
&2\log\Phib(y)-\log\Phib(y+h)-\log\Phib(y-h)\cr
&\leq 2\psib\left(\frac{(|y|^2+ |h|^2+ 2\langle y,h\rangle)+ (|y|^2+|h|^2-2\langle y,h\rangle) }{2}\right) -2\psib(|y|^2)\cr
&=2\left[\psib(|y|^2+ |h|^2)-\psib(|y|^2)\right]\cr
&\leq 2\psib(|h|^2)\cr
&=2\log\Phib(0)-\log\Phib(h)-\log\Phib(-h),
\end{align*}
where we also used the radial symmetry in the last equality. Multiplying by the term $c_{\beta,d}/(2|h|^{d+\beta})$ and integrating in $h$ leads to
\begin{align*}
L_\beta(\log\Phib)(y)&\leq L_\beta(\log\Phib)(0)\cr
 &=\frac{c_{\beta,d}}{2}\int_{\R^d}\frac{2\log\Phib(0)-\log\Phib(h)-\log\Phib(-h)}{|h|^{d+\beta}}\dd h  \\
&=c_{\beta,d}\int_{\R^d}\frac{\psib(|h|^2)}{|h|^{d+\beta}}\dd h.
\end{align*}
Taking the supremum over $y\in\R^d$ proves that the supremum in \eqref{CLY-Abeta} is attained at $y=0$, and \eqref{sup-origin-1} follows.
\end{proof}

Now we give a decomposition of the constant $\CLY(\beta,d)$ via comparing $L_\beta(\log\Phib)(0)$ with $L_\beta\Phib(0)/\Phib(0)$. For $a\geq0$, set
$$ H(a):=a-1+e^{-a}.$$
\begin{lemma}\label{main-decomposition}
For every $\beta\in(0,2)$,
\begin{equation}\label{decomp}
       \CLY(\beta,d)=\frac d\beta+\mathcal R_{\beta,d},
\end{equation}
where
\begin{equation}\label{R-def}
       \mathcal R_{\beta,d}
       =c_{\beta,d}\int_{\R^d}
       \frac{H\bigl(\psib(|h|^2)\bigr)}{|h|^{d+\beta}}\dd h.
\end{equation}
Moreover, $\mathcal R_{\beta,d}>0$ for any $\beta\in(0,2)$.
\end{lemma}
\begin{proof}
By  \eqref{self-sim},
$$ G^{(\beta)}(t,0)=t^{-d/\beta}\Phib(0).$$
Thus,
$$\partial_t\big|_{t=1}G^{(\beta)}(t,0) =-\frac d\beta\Phib(0).$$
Since $\Phib(x)=G^{(\beta)}(1,x)$, and $G^{(\beta)}$ solves the fractional heat equation
\begin{equation*}\label{heat-kernel-eq}
       \partial_tG^{(\beta)}(t,x)=-L_\beta G^{(\beta)}(t,\cdot)(x),
\end{equation*}
we have
\begin{equation}\label{eq:Abeta-Phi-identity}
       \frac{L_\beta\Phib(0)}{\Phib(0)}=\frac{-\partial_t|_{t=1}G^{(\beta)}(t,0)}{\Phib(0)}=\frac d\beta.
\end{equation}
On the other hand, by \eqref{Abeta} and the radial symmetry of $\Phib$, we obtain
\begin{align}\label{eq:Abeta-Phi-integral}
       \frac{L_\beta\Phib(0)}{\Phib(0)}
       &=c_{\beta,d}\int_{\R^d}
       \frac{\Phib(0)-\Phib(h)}{\Phib(0)|h|^{d+\beta}}\dd h\notag\\
       &=c_{\beta,d}\int_{\R^d}
       \frac{1-e^{-\psib(|h|^2)}}{|h|^{d+\beta}}\dd h,
\end{align}
where the last equality applied \eqref{eq-Phi-psi}. Combining \eqref{eq:Abeta-Phi-identity} and \eqref{eq:Abeta-Phi-integral}, we obtain
\begin{equation*}
       \frac d\beta
       =c_{\beta,d}\int_{\R^d}
       \frac{1-e^{-\psib(|h|^2)}}{|h|^{d+\beta}}\dd h.
\end{equation*}
Combing this with Proposition \ref{sup-origin} yields \eqref{decomp} and \eqref{R-def}.

Finally, it is easy to observe that $H(0)=0$ and $H'(a)=1-e^{-a}\geq0$ for $a\geq0$, so $H(a)\geq0$ on $[0,\infty)$.  Since $\psib(|h|^2)>0$ for $h\ne0$ and $H(a)>0$ for $a>0$, we have $\mathcal R_{\beta,d}>0$.
\end{proof}

In the sequel, we shall repeatedly use the elementary estimate
\begin{equation}\label{H-estimate}
0\leq H(a) \leq \min\left\{a,\frac{a^2}{2}\right\},\quad a\geq0.
\end{equation}
Indeed, for $a\geq0$, $H(a)\leq a$ is clear, and $H(a)\leq a^2/2$ follows from the elementary inequality $e^{-a}\leq1-a+a^2/2$.

The next result establish an upper bound on the normalizing constant $c_{\beta,d}$ with a factor $2-\beta$.
\begin{lemma}\label{cbeta}
There exists a constant $K_d>0$, depending only on $d$, such that
\begin{equation*}\label{cbeta-bound}
       c_{\beta,d}\leq K_d(2-\beta),\quad 1\leq\beta<2.
\end{equation*}
\end{lemma}

\begin{proof}
By using the property of the Gamma function $\Gamma$ (see \eqref{gamma}), we get
\begin{align*}
       c_{\beta,d}
       &=\frac{2^\beta\Gamma\bigl((d+\beta)/2\bigr)}{\pi^{d/2}|\Gamma(-\beta/2)|}  \\
       &=\frac{\beta 2^{\beta-1}\Gamma\bigl((d+\beta)/2\bigr)}{\pi^{d/2}\Gamma(1-\beta/2)}  \\
       &=(2-\beta)\frac{\beta 2^{\beta-2}\Gamma\bigl((d+\beta)/2\bigr)}{\pi^{d/2}\Gamma(2-\beta/2)}\\
       &=:(2-\beta) U_d(\beta).
\end{align*}
The function $\beta\mapsto U_d(\beta)$ is continuous on the closed interval $[1,2]$, so it is bounded. By letting $K_d=\sup_{\beta\in [1,2)}U_d(\beta)$, we finish the proof.
\end{proof}

The following lower bound on $F_\beta$, which may be of independent interest, is precisely where the factor $2-\beta$ must be kept.  Standard two-sided estimates for stable heat kernels often hide this factor in constants, which is not sufficient for the limiting problem considered here. For a comprehensive introduction to Poisson random measures and the L\'evy--It\^o decomposition, see \cite[Chapter 4]{Sato1999} and \cite[Sections 2.3 and 2.4]{Appl2009}.
\begin{lemma}\label{tail-lower}
There exists $k_d>0$, depending only on $d$, such that
\begin{equation*}
F_\beta(r)\geq k_d(2-\beta)r^{-d-\beta},\quad r\geq1,\,1\leq\beta<2.
\end{equation*}
\end{lemma}
\begin{proof}
Let \(\mathcal N_\alpha\) be a Poisson random measure on $(0,1]\times(0,\infty)$ with intensity measure ${\rm d}t\,\nu_\alpha({\rm d} s),$ where $\nu_\alpha$ is given by \eqref{levy-mass}. By the L\'evy--It\^o decomposition (see \cite[Theorem 2.4.16]{Appl2009}), we can write $S_1^\alpha=Y_\alpha+Z_\alpha$ almost surely, where the small-jump part $Y_\alpha$ and the large-jump part $Z_\alpha$ are given by
$$Y_\alpha:=\int_{(0,1]\times(0,1]}s\,\mathcal N_\alpha({\rm d} t,{\rm d} s),\quad Z_\alpha:=\int_{(0,1]\times(1,\infty)}s\,\mathcal N_\alpha({\rm d} t,{\rm d} s).$$

We now turn to estimate $Y_\alpha$.
\begin{align*}
\E Y_\alpha= \int_0^1 s\,\nu_\alpha({\rm d}s)
&=\frac{\alpha}{\Gamma(1-\alpha)}\int_0^1 s^{-\alpha}\dd s\\
&=\frac{\alpha}{(1-\alpha)\Gamma(1-\alpha)}=\frac{\alpha}{\Gamma(2-\alpha)}.
\end{align*}
As a function of $\alpha$, the right-hand side is bounded on $[1/2,1]$. Then, there exists a constant $C>0$ (independent of $\alpha$) such that
$$\E Y_\alpha\leq C,\quad \alpha\in[1/2,1).$$
Choose $M=2C$.  By Markov's inequality,
$$\Pb(Y_\alpha>M)\leq\frac{\E Y_\alpha}{M}\leq \frac12.$$
Hence,
\begin{equation}\label{Y-M}
\Pb(Y_\alpha\leq M)\geq \frac12,\quad \alpha\in[1/2,1).
\end{equation}

Next we consider $Z_\alpha$.  The total intensity of jumps larger than $1$ is
\begin{align}\label{chi-alpha}
\chi_\alpha:=\nu_\alpha((1,\infty))&=\frac{\alpha}{\Gamma(1-\alpha)}\int_1^\infty s^{-1-\alpha}\dd s\cr
&=\frac1{\Gamma(1-\alpha)}=\frac{1-\alpha}{\Gamma(2-\alpha)},
\end{align}
where the last term is uniformly bounded on $[1/2,1)$. In particular, the number of jumps larger than $1$ is finite almost surely, and $Z_\alpha$ is a compound Poisson random variable.

Fix $r\geq1$, and let
$$D_r=(r^2,2r^2],\quad B_r=(1,\infty)\setminus D_r.$$
Set
$$N_1:=\mathcal N_\alpha((0,1]\times D_r),\quad N_2:=\mathcal N_\alpha((0,1]\times B_r).$$
Then $N_1$ and $N_2$ are Poisson random variables with intensities $\nu_\alpha(D_r)$ and $\nu_\alpha(B_r)$, respectively.  We define the event
$$E_{\alpha,r}:= \{Y_\alpha\leq M\} \cap \{ N_1=1\}\cap\{ N_2=0\}.$$
Since the three sets $(0,1]\times(0,1]$, $(0,1]\times D_r$ and $(0,1]\times B_r$ are pairwise disjoint, by the independence property of Poisson random measures, the events $\{Y_\alpha\leq M\}$, $\{N_1=1\}$ and $\{N_2=0\}$ are independent.  Hence
\begin{align}\label{estmate-E}
\Pb(E_{\alpha,r}) &=\Pb(Y_\alpha\leq M)\Pb(N_1=1)\Pb(N_2=0) \cr
&= \Pb(Y_\alpha\leq M) e^{-\nu_\alpha(D_r)}\nu_\alpha(D_r)e^{-\nu_\alpha(B_r)} \cr
&=\Pb(Y_\alpha\leq M)e^{-[\nu_\alpha(D_r)+\nu_\alpha(B_r)]}\nu_\alpha(D_r).
\end{align}
Observe that $\nu_\alpha(D_r)+\nu_\alpha(B_r)=\nu_\alpha((1,\infty))=\chi_\alpha$, since $D_r\cup B_r=(1,\infty)$ and $D_r\cap B_r=\emptyset$, and
\begin{align*}
\nu_\alpha(D_r)&=\frac{\alpha}{\Gamma(1-\alpha)}\int_{r^2}^{2r^2}s^{-1-\alpha}\dd s\\
&=\frac{1-2^{-\alpha}}{\Gamma(1-\alpha)}r^{-2\alpha}=\frac{(1-2^{-\alpha})(1-\alpha)}{\Gamma(2-\alpha)}r^{-2\alpha}.
\end{align*}
Combining this with \eqref{estmate-E}, \eqref{Y-M} and \eqref{chi-alpha}, we obtain
\begin{align*}
\Pb(E_{\alpha,r})&=\Pb(Y_\alpha\leq M)e^{-\chi_\alpha}\nu_\alpha(D_r)\\
 &\geq \frac12 \frac{(1-2^{-\alpha})(1-\alpha)}{\Gamma(2-\alpha)} e^{-\sup_{\alpha\in[1/2,1)}\chi_\alpha} r^{-2\alpha}.
\end{align*}
Thus, there exists a constant $c>0$, independent of $\alpha$ and $r$, such that
\begin{equation}\label{lower-estimate-E}
\Pb(E_{\alpha,r})\geq c(1-\alpha)r^{-2\alpha}.
\end{equation}

Note that on $E_{\alpha,r}$, there is exactly one jump $J(\omega)\in D_r$ larger than $1$. Consequently, for every $\omega\in E_{\alpha,r}$,
$$r^2<S_1^\alpha(\omega)=Y_\alpha(\omega)+J(\omega)\leq 2r^2+M\leq(2+M)r^2,\quad r\geq1.$$
Combining this with \eqref{lower-estimate-E}, we can find a constant $c_d>0$, depending only on $d$, such that
\begin{align*}
F_\beta(r)&=\E\left[(4\pi S_1^\alpha)^{-d/2}\exp\left(-\frac{r^2}{4S_1^\alpha}\right)\right] \\
&\geq\E\left[(4\pi S_1^\alpha)^{-d/2}\exp\left(-\frac{r^2}{4S_1^\alpha}\right)\mathbf 1_{E_{\alpha,r}} \right] \\
&\geq[4\pi(M+2)]^{-d/2}e^{-1/4}r^{-d}\Pb(E_{\alpha,r})\\
&\geq c_d(1-\alpha)r^{-d-2\alpha},\quad r\geq1.
\end{align*}
Since $1-\alpha=(2-\beta)/2$, the desired estimate follows.
\end{proof}

As a direct consequence of Lemma \ref{tail-lower}, we have the following result.
\begin{lemma}\label{psi-tail}
There exists a constant $B_d>0$, depending only on $d$, such that for $r\geq1$ and $1\leq\beta<2$,
\begin{equation*}
0\leq \psib(r^2)\leq B_d+\log\frac1{2-\beta}+(d+\beta)\log r.
\end{equation*}
\end{lemma}
\begin{proof}
The lower bound follows from Lemma \ref{psi-concave}(i). For the upper bound, by \eqref{negative-moment}, we have the uniform bound
\begin{align*}
F_\beta(0)&=(4\pi)^{-d/2}\E (S_1^\alpha)^{-d/2}  \\
       &=(4\pi)^{-d/2}\frac{\Gamma(d/\beta)}{(\beta/2)\Gamma(d/2)}\leq C_d,\quad 1\leq\beta<2,
\end{align*}
where $C_d>0$ is a constant depending only on $d$. Combining this with Lemma~\ref{tail-lower}, for any $r\geq1$ and $1\leq\beta<2$,
\begin{align*}
\psib(r^2)=\log\frac{F_\beta(0)}{F_\beta(r)} &\leq \log\frac{C_d}{k_d(2-\beta)r^{-d-\beta}}\\
&=B_d+\log\frac1{2-\beta}+(d+\beta)\log r,
\end{align*}
where $k_d>0$ is a constant depending only on $d$ and $B_d=\log(C_d/k_d)$.
\end{proof}

\section{Proof of the main theorem}\label{sec-proof}\hskip\parindent
Building on the results of the previous section, we now proceed to prove our main result.
\begin{proof}[Proof of Theorem \ref{main}]
Split the remainder $\mathcal R_{\beta,d}$ in \eqref{decomp} into two parts:
\begin{align}\label{pf-main-1}
\mathcal R_{\beta,d} &=c_{\beta,d}\int_{|h|<1}\frac{H\bigl(\psib(|h|^2)\bigr)}{|h|^{d+\beta}}\dd h
+c_{\beta,d}\int_{|h|\geq1}\frac{H\bigl(\psib(|h|^2)\bigr)}{|h|^{d+\beta}}\dd h\cr
&=: {\rm I} + {\rm II}.
\end{align}

We first estimate ${\rm I}$. By \eqref{H-estimate} and Lemma \ref{psi-concave} (v), we have
\begin{align*}
H\big(\psib(|h|^2)\big)\leq \frac12\psib(|h|^2)^2\leq \frac{M_d^2}{2}|h|^4,\quad |h|<1.
\end{align*}
Hence, using polar coordinates and Lemma~\ref{cbeta}, we derive
\begin{align}\label{pf-main-2}
{\rm I}&\leq \frac{c_{\beta,d}M_d^2}{2}|\sphere|\int_0^1 r^{3-\beta}\dd r \cr
&=\frac{c_{\beta,d}M_d^2}{2(4-\beta)}|\sphere|\cr
&\leq C_d(2-\beta),\quad \beta\in[1,2),
\end{align}
for some constant $C_d>0$ depending only on $d$.

Now we estimate ${\rm II}$. Let $\eps=2-\beta$. By \eqref{H-estimate} and  Lemma~\ref{psi-tail}, we have
\begin{align}\label{pf-main-II}
{\rm II}&\leq c_{\beta,d}|\sphere|\int_1^\infty\frac{B_d-\log\eps+(d+\beta)\log r}{r^{1+\beta}}\dd r\cr
&=c_{\beta,d}|\sphere|\left[\frac{B_d-\log\eps}{\beta}+\frac{d+\beta}{\beta^2}\right],
\end{align}
where in the last equality we also used the elementary identities
$$\int_1^\infty r^{-1-b}\dd r=\frac1b,\quad \int_1^\infty (\log r)r^{-1-b}\dd r=\frac1{b^2},\quad\quad b>0.$$
Since $\beta\in[1,2)$ and $c_{\beta,d}\leq K_d\eps$ (see Lemma \ref{cbeta}), by \eqref{pf-main-II}, we obtain
\begin{align}\label{pf-main-3}
{\rm II}&\leq K_d|\sphere| \eps\left(B_d+d+\log\frac e\eps\right)\cr
&\leq \tilde{K}_d\eps\log\frac e\eps
=\tilde{K}_d(2-\beta)\log\frac{e}{2-\beta},\quad \beta\in[1,2),
\end{align}
where $K_d,\tilde{K}_d>0$ are constants depending only on $d$.

Putting \eqref{pf-main-1}, \eqref{pf-main-2} and \eqref{pf-main-3} together, we have for any $1\leq\beta<2$,
\begin{align*}
\mathcal R_{\beta,d}&\leq C_d(2-\beta)+\tilde{K}_d(2-\beta)\log\frac e{2-\beta}\\
&\leq M_d (2-\beta)\log\frac e{2-\beta},
\end{align*}
where $M_d>0$ is a constant depending only on $d$.  Combining this with Lemma \ref{main-decomposition} gives \eqref{main-rate}. By letting $\beta\rightarrow2^-$, we immediately obtain \eqref{main-limit}.
\end{proof}


\end{document}